\documentclass[ejs,preprint]{imsart}

\RequirePackage[OT1]{fontenc}
\RequirePackage{amsthm,amsmath}
\RequirePackage[numbers]{natbib}
\RequirePackage[colorlinks,citecolor=blue,urlcolor=blue]{hyperref}

\RequirePackage{amsmath,amsthm,amssymb,amsbsy}
\RequirePackage{graphicx}

\doi{10.1214/14-EJS932}
\pubyear{2014}
\volume{8}
\issue{0}
\firstpage{1438}
\lastpage{1459}

\startlocaldefs
\numberwithin{equation}{section}
\theoremstyle{plain}
\newtheorem{theorem}{Theorem}[section]
\newtheorem{lemma}[theorem]{Lemma}

\newtheorem{corollary}[theorem]{Corollary}
\theoremstyle{definition}

\newtheorem{example}[theorem]{Example}
\theoremstyle{remark}

\endlocaldefs

\newcommand{\FF}{\mathbb{F}}

\newcommand{\MM}{\mathbb{M}}

\newcommand{\PP}{\mathbb{P}}
\newcommand{\QQ}{\mathbb{Q}}
\newcommand{\RR}{\mathbb{R}}

\newcommand{\BBB}{\mathcal{B}}

\newcommand{\FFF}{\mathcal{F}}

\newcommand{\LLL}{\mathcal{L}}

\newcommand{\NNN}{\mathcal{N}}

\newcommand{\PPP}{\mathcal{P}}

\newcommand{\dv}{\,\mathrm{d}}                              

\newcommand{\eps}{\varepsilon}                              
\newcommand{\foralls}{\;\forall\;}                          

\newcommand{\cov}{\mathrm{Cov}}                             

\begin{document}

\begin{frontmatter}
\title{Quantifying identifiability in independent component analysis}
\runtitle{Quantifying identifiability in ICA}

\begin{aug}

\author{\fnms{Alexander} \snm{Sokol}*\ead[label=e1]{alexander@math.ku.dk}}

\address{Institute of Mathematics \\
University of Copenhagen \\
2100 Copenhagen, Denmark \\
\printead{e1}}

\author{\fnms{Marloes} \snm{H. Maathuis}\ead[label=e2]{maathuis@stat.math.ethz.ch}}

\address{Seminar f\"{u}r Statistik \\
ETH Z\"{u}rich \\
8092 Z\"{u}rich, Switzerland\\
\printead{e2}}

\author{\fnms{Benjamin} \snm{Falkeborg} \ead[label=e3]{benjamin.falkeborg@econ.ku.dk}}

\address{Department of Economics \\
University of Copenhagen\\
2100 Copenhagen, Denmark\\
\printead{e3}}

\runauthor{A. Sokol et al.}

\end{aug}

\begin{abstract}
We are interested in consistent estimation of the mixing matrix in the ICA
model, when the error distribution is close to (but different from)
Gaussian. In particular, we consider $n$ independent samples from the ICA
model $X = A\epsilon$, where we assume that the coordinates of $\epsilon$
are independent and identically distributed according to a contaminated Gaussian distribution, and the amount of contamination is allowed
to depend on $n$. We then investigate how the ability to consistently estimate the mixing matrix
depends on the amount of contamination. Our results suggest that in an asymptotic sense, if the amount of contamination
decreases at rate $1/\sqrt{n}$ or faster, then the mixing matrix is only identifiable up to transpose products. These results also have
implications for causal inference from linear structural equation
models with near-Gaussian additive noise.
\end{abstract}




\end{frontmatter}

\section{Introduction}

We consider the $p$-dimensional independent component analysis (ICA) model 
\begin{align}
  X &= A\epsilon,\label{eq:SEM}
\end{align}
where $A$ is a $p\times p$ mixing matrix, $\epsilon$ is a $p$-dimensional error (or source)
variable with independent and nondegenerate coordinates of mean zero, and $X$ is a $p$-dimensional observational variable.
Based on observations of $X$, ICA aims to estimate the mixing matrix $A$
and the distribution of the error variable $\epsilon$. Theory and
algorithms for ICA can be found in, e.g.,
\cite{MR2329469,PC1994,AH1999,AH2013,HKO2001,NA2008,YuanSamworth2012}. ICA has applications
in many different disciplines, including blind source separation (e.g.,
\cite{CJ2010}), face recognition (e.g., \cite{MSB2002}), medical imaging
(e.g., \cite{CFB2004, JEA2001, VSJHO2000}) and causal discovery using the
LiNGAM method (e.g., \cite{MR2274431,MR2804599}).  

Our focus is on consistent estimation of the mixing matrix. Here, identifiability is an issue,
since two different mixing matrices $A$ and $B$ may yield the same
distribution of $X$. This is the case, for example, if the distribution of $\epsilon$ is
multivariate Gaussian and $AA^t = BB^t$. In this case, the mixing matrix cannot be identified
from $X$, instead only the transpose product of the mixing matrix can be identified. In \cite{PC1994}, it was shown that whenever at most one of the
components of $\epsilon$ is Gaussian, the mixing matrix is identifiable up to
scaling and permutation of columns. This result was expanded upon in the ICA context in Theorem 4 of \cite{EK2004} and
extended considerably to the broader class of additive index models in Theorem 1 of \cite{Yuan2011}. In order to illustrate
the relevance of the mixing matrix in (\ref{eq:SEM}), we give an example based on causal inference.

\begin{example}
\label{example:ICACausal}
Consider a two-dimensional linear structural equation model with
additive noise of the form
\begin{align}
  \left(\begin{array}{c}
    X_1 \\ X_2 \end{array}\right)
  &=\left(\begin{array}{cc}
    C_{11} & C_{12} \\ C_{21} & C_{22} \end{array}\right)
    \left(\begin{array}{c}  X_1 \\ X_2 \end{array}\right)
  +
  \left(\begin{array}{c}
    \epsilon_1 \\ \epsilon_2 \end{array}\right)\label{eq:TwoDimASem},
\end{align}
see, e.g., \cite{MR2274431}. We assume that the coordinates of $\epsilon$
are independent, nondegenerate and have mean zero, and are independent of $X$. We also assume that
$C$ is strictly triangular, meaning that all entries of $C$ are zero
except either $C_{12}$ or $C_{21}$. In the context of linear structural
equation models,  identifying $C$ corresponds to identifying whether $X_1$ is a cause of $X_2$ (corresponding to $C_{21}\neq0$), $X_2$ is
a cause of $X_1$ (corresponding to $C_{12}\neq0$), or neither is a cause of the other (corresponding to $C_{21} = C_{12} = 0$).

As $C$ is strictly triangular, $I-C$ is invertible. Letting
$A=(I-C)^{-1}$, we obtain
\begin{align}
  \left(\begin{array}{c}
    X_1 \\ X_2 \end{array}\right)
  &=\left(\begin{array}{cc}
    A_{11} & A_{12} \\ A_{21} & A_{22} \end{array}\right)
  \left(\begin{array}{c}
    \epsilon_1 \\ \epsilon_2 \end{array}\right)\label{eq:TwoDimAICA},
\end{align}
where $A$ is upper or lower triangular according to whether the same
holds for $C$. Thus, we have arrived at an ICA model of the form
(\ref{eq:SEM}). In the case where $\epsilon$ is jointly Gaussian, it is immediate that we cannot
identify $A$ from the distribution of $X$ alone. By the results of \cite{PC1994,MR2274431},
identification of $A$ up to scaling and permutation of columns from the
distribution of $X$ is possible when $\epsilon$ has at most one Gaussian coordinate. In this case, we may
therefore infer causal relationships from estimation of the mixing matrix in an
ICA model. The LiNGAM method \cite{MR2274431} is based on this idea. However, if we only get $n$ i.i.d.\ samples from $X$
and the distribution of $\epsilon$ is close to Gaussian, it may be expected that estimation of $A$ becomes difficult,
and consequently causal inference is difficult as well.
\hfill$\circ$
\end{example}

Motivated by the above, we take an interest in the following question: When the distribution of $\epsilon$ is
close to Gaussian but non-Gaussian, how difficult is it to consistently estimate the mixing matrix as the number $n$ of observations
increases? For any fixed non-Gaussian $\epsilon$, the results of \cite{PC1994} show that $A$ is identifiable up to
scaling and permutation of columns, so letting $n$ tend to infinity in this model, we should be able to
consistently estimate $A$ up to such scaling and permutation of columns. However, if the distribution of $\epsilon$ converges to a Gaussian distribution, it is natural
to expect that the model begins to take on the properties of ICA models with Gaussian errors, where only the transpose
product $AA^t$ is identifiable. Our question, heuristically speaking, is: How large should the number of observations $n$
be in order to counterbalance the near-Gaussianity of the noise distribution? We think of this as quantifying the level
of identifiability of $A$. In order to elucidate this issue,
we will consider asymptotic scenarios where the distribution of $\epsilon$ depends on the sample size $n$, and tends to a
Gaussian distribution as $n$ tends to infinity.

\section{Problem statement and main results}

\label{sec:ProbState}

ICA can be used to estimate $A$
when the distribution of $\eps$ is unknown. In this case, we may think
of the statistical model (\ref{eq:SEM}) as the collection of
probability measures
\begin{align}
  \{L_A(R)\mid A\in\MM(p,p),R\in\PPP(p)\},\label{eq:ICASemiparamModel}
\end{align}
where $\MM(p,p)$ denotes the space of $p\times p$ matrices, $L_A:\RR^p\to\RR^p$ is given
by $L_A(x)=Ax$, $L_A(R)$ denotes the image measure of $R$ under the
transformation $L_A$, $\PPP(p)$ denotes the set of product probability
measures on $(\RR^p,\BBB_p)$ with nondegenerate mean zero coordinates and $\BBB_p$ denotes the Borel $\sigma$-algebra on
$\RR^p$. With $\eps$ having distribution $R$,
this means that the error distribution has independent nondegenerate
mean zero coordinates. In other words, it is assumed that the distribution
of $X$ in (\ref{eq:SEM}) is equal to $L_A(R)$ for some $A\in\MM(p,p)$
and $R\in\PPP(p)$. This is a semiparametric model, where $A$ is the parameter of interest and 
$R$ is a nuisance parameter. Asymptotic distributions of estimates
of the mixing matrix in this type of set-up are derived in, e.g., \cite{AC1997, MR2329469, MR2906874}. 
The difficulty of estimating $A$ can then be appraised by considering for example the asymptotic variance of the
estimates. 

Alternatively, one can consider estimation of $A$ for a
given error distribution. This is the approach we take in this paper. When
$\eps$ has the distribution of some fixed $R\in\PPP(p)$, the statistical model (\ref{eq:SEM}) is
the collection of probability measures
\begin{align}
  \{L_A(R)\mid A\in\MM(p,p)\}.\label{eq:ICAModel}
\end{align}
Results on identifiability of $A$ in (\ref{eq:ICAModel}) follow from
the results of \cite{PC1994} and \cite{EK2004}. In particular, if no two coordinates of $R$ are jointly 
Gaussian, the mixing matrix $A$ is identifiable up to
sign reversion and permutation of columns, in the sense that $L_A(R)=L_B(R)$ implies
$A=B\Lambda P$ for some diagonal matrix $\Lambda$ with $\Lambda^2=I$
and some permutation matrix $P$.

We are interested in how difficult it is to consistently estimate the mixing matrix in
(\ref{eq:ICAModel}) when the error distributions are different from
  Gaussian but close to Gaussian. Some results in this direction can be
found in \cite{OKK2008}, where the authors calculated the Cr{\'a}mer-Rao lower bound for
the model (\ref{eq:ICAModel}), under the assumption that the coordinates of
the error distribution have certain regularity criteria such as finite
variance and differentiable Lebesgue densities. These results indicate how the minimum variance of an unbiased
estimator of the mixing matrix depends on the error distribution.

We consider the problem from the following different perspective. For
$p\ge1$ and any signed measure $\mu$ on $(\RR,\BBB)$, let $\mu \otimes \mu$ denote the product measure of $\mu$ with itself, and let
$\mu^{\otimes p} = \otimes_{i=1}^p \mu$ denote the $p$-fold product
measure. Fix two nondegenerate mean zero probability measures $\xi$ and $\zeta$ with $\xi\neq\zeta$,
and let $P_e(\beta)$ be the contaminated distribution given by
\begin{align}
  P_e(\beta) = \beta \xi + (1-\beta)\zeta.
\end{align}
We write $F^A_\beta$ for the cumulative
distribution function of $L_A(P_e(\beta)^{\otimes p})$, and write
$F^A=F^A_0$. Note that $F^A$ is then the cumulative distribution function of
$L_A(\zeta^{\otimes p})$. We assume that we observe
$n$ i.i.d.\ observations from the distribution $F_{\beta_n}^A$, where the amount of contamination $\beta_n$ is allowed
to depend on the sample size. Our results indicate that, in this framework, consistent estimation of $A$
(up to scaling and permutation of columns) cannot be expected when $\beta_n = o(1/\sqrt{n})$ and $\zeta$ is mean zero Gaussian.

\section{An upper asymptotic distance bound}

\label{sec:UpperBound}

In this section, we develop some preliminary results which will be used to prove our main results in
Section \ref{sec:Identifiability}. We begin by introducing some notation. For any measure $\mu$
on $(\RR^p,\BBB_p)$, let $|\mu|$ denote the total
variation measure of $\mu$, see, e.g., \cite{MR924157}. We define two norms by
\begin{align}
  \|\mu\|_\infty &=\sup_{x\in\RR^p}|\mu((-\infty,x_1]\times\cdots\times(-\infty,x_p])|,\\
  \|\mu\|_{tv} &= |\mu|(\RR^p),
\end{align}
and refer to these as the uniform and the total variation norms,
respectively. The uniform norm for measures is also known as the
Kolmogorov norm. Note that if $P$ and $Q$ are two probability measures on
$(\RR^p,\BBB_p)$ with cumulative distribution functions $F$ and $G$,
then $\|P-Q\|_\infty=\|F-G\|_\infty$. Finally, we use the notation
$f(s) \sim g(s)$ as $s\to s_0$ when $\lim_{s\to s_0} f(s)/g(s) =
1$.

As in the previous section, let $\xi$ and $\zeta$ be two nondegenerate mean zero probability
distributions on $(\RR,\BBB)$ with $\xi\neq\zeta$. We aim to bound the
distance
\begin{align}
  \|F_{\beta}^A - F_{\beta}^B\|_{\infty} = \|L_{A}(P_e(\beta)^{\otimes p})-L_{B}(P_e(\beta)^{\otimes p})\|_{\infty}  
\end{align}
for matrices $A,B\in\MM(p,p)$ with $F^A=F^B$. To this end, define
\begin{align}
  \nu&=(\xi-\zeta)/\|\xi-\zeta\|_\infty.
\end{align}
The following theorem is a first step towards our goal.


\begin{theorem}
\label{theorem:GrossErrorAsymptotics}
Let $\beta\in (0,1)$, and let
$A\in\MM(p,p)$. Then
\begin{align}
  \lim_{\beta\to0}
   \frac{L_{A}(P_e(\beta)^{\otimes p})-L_{A}(\zeta^{\otimes p})}{\|P_e(\beta)-\zeta\|_\infty}
   =\sum_{k=1}^p L_{A}(\zeta^{\otimes(k-1)}\otimes\nu\otimes\zeta^{\otimes(p-k)})\label{eq:GrossErrorLimit},
\end{align}
where convergence is in
$\|\cdot\|_\infty$. Moreover, $F^A_\beta$ tends uniformly to $F^A$ as
$\beta$ tends to zero.
\end{theorem}

The proof of Theorem \ref{theorem:GrossErrorAsymptotics} exploits
properties of the contaminated distributions $P_e(\beta)$ for $\beta\in(0,1)$, in
particular the fact that $\|P_e(\beta)-\zeta\|_\infty$ is nonzero and linear in
$\beta$ and that $(P_e(\beta)-\zeta)/\|P_e(\beta)-\zeta\|_\infty$ is
constant in $\beta$. These properties are used to obtain a decomposition of $P_e(\beta)^{\otimes p}$ as a
polynomial function of $\beta$ in the proof. As Lemma \ref{lemma:OnlyContaminatedProps} shows,
only contaminated distributions have these properties. This is our
main reason for working with this family of distributions. 

\begin{lemma}
\label{lemma:OnlyContaminatedProps}
Let $\beta\mapsto Q(\beta)$ be a mapping from $(0,1)$ to the space of
probability measures on $(\RR,\BBB)$ with the properties that
$\|Q(\beta)-\zeta\|_\infty$ is nonzero and linear in $\beta$ and
$(Q(\beta)-\zeta)/\|Q(\beta)-\zeta\|_\infty$ is constant in
$\beta$. Then $Q(\beta)$ can be written as a contaminated
$\zeta$ distribution, in the sense that $Q(\beta)=\beta\xi+(1-\beta)\zeta$ for
some probability measure $\xi$ on $(\RR,\BBB)$.
\end{lemma}

Due to the properties of contaminated distributions, Theorem
\ref{theorem:GrossErrorAsymptotics} in fact also holds for other norms than
the uniform norm. However, the choice of the norm is important when we wish
to bound the norm of the right-hand side of
(\ref{eq:GrossErrorLimit}). Such a bound is achieved in Lemma
\ref{lemma:KNormBound}.

\begin{lemma}
\label{lemma:KNormBound}
  Let $A \in \MM(p,p)$. Then 
\begin{align}
  \left\|\sum_{k=1}^p L_{A}\left(\zeta^{\otimes(k-1)}\otimes\nu\otimes\zeta^{\otimes(p-k)}\right)\right\|_\infty
  \le 2p.
\end{align}
\end{lemma}

Combining Theorem \ref{theorem:GrossErrorAsymptotics} and Lemma
\ref{lemma:KNormBound} yields the following corollary, which we give
without proof. 

\begin{corollary}
\label{corr:TotalNormBound}
Let $A,B\in\MM(p,p)$ be such that $F^A=F^B$. Define
\begin{align}
  \Gamma(A,B,\nu)=&\sum_{k=1}^p L_{A}\left(\zeta^{\otimes(k-1)}\otimes\nu\otimes\zeta^{\otimes(p-k)}\right)\notag\\
                 -&\sum_{k=1}^p L_{B}\left(\zeta^{\otimes(k-1)}\otimes\nu\otimes\zeta^{\otimes(p-k)}\right).\label{eq:DefGamma}
\end{align}
Then we have, for $\beta \to 0$,
\begin{align}
   \| F_{\beta}^A - F_{\beta}^B \|_{\infty} & \sim \left\| \Gamma\left(A,B,\frac{\xi - \zeta}{\|\xi - \zeta\|_{\infty}}\right) \right\|_{\infty} \|P_e(\beta)-\zeta\|_{\infty} 
      \le 4p\beta \|\xi-\zeta\|_{\infty}.
\end{align}
\end{corollary}
Corollary \ref{corr:TotalNormBound} shows that, if $F^A=F^B$, the distance between the observational distributions
$F_{\beta}^A$ and $F_{\beta}^B$ decreases asymptotically linearly in $\beta$ as $\beta$ tends to zero.

The corollary is stated under the condition that
$F^A=F^B$. For later use, we characterize the occurrence of this in the
next lemma, in terms of $\zeta$, $A$ and $B$, for the case where $A$ and $B$ are invertible. Recall that a probability distribution $Q$ on
$(\RR,\BBB)$ is said to be symmetric if, for every random variable $Y$
with distribution $Q$, $Y$ and $-Y$ have the same distribution. The proof
of Lemma \ref{lemma:GaussianUniqueness} is a simple consequence of Theorem 4 of \cite{EK2004}.

\begin{lemma}
\label{lemma:GaussianUniqueness}
Let $A,B \in \MM(p,p)$ be invertible. Then the following hold:
\begin{enumerate}
\item If $\zeta$ is Gaussian, then $F^A = F^B$ if and only
  if $AA^t=BB^t$.
\item If $\zeta$ is non-Gaussian and symmetric, then $F^A = F^B$ if and only
  if we have $A=B\Lambda P$ for some permutation matrix $P$ and a diagonal matrix
  $\Lambda$ satisfying $\Lambda^2=I$.
\item If $\zeta$ is non-symmetric, then $F^A = F^B$ if and only if $A=B P$ for some permutation matrix $P$.
\end{enumerate}
\end{lemma}

\section{Asymptotic properties of ICA models with near-Gaussian noise}

\label{sec:Identifiability}




We now use the results obtained in Section \ref{sec:UpperBound} to obtain asymptotic results on ICA
models with near-Gaussian noise. We will consider a sequence of ICA models with increasingly
near-Gaussian noise, and will investigate the asymptotic properties of this sequence of models.

We need some basic facts about random fields in order to formulate our results, see
\cite{MR1914748} and \cite{MR1472736} for an overview. Recall that a
mapping $R:\RR^p\times\RR^p\to\RR$ is said to be symmetric if
$R(x,y)=R(y,x)$ for all $x,y\in\RR^p$, and is said to be positive
semidefinite if for all $n\ge1$ and for all $x_1,\ldots,x_n\in\RR^p$ and
$\xi_1,\ldots,\xi_n\in\RR$, it holds that
\begin{align}
  \sum_{i=1}^n\sum_{j=1}^n\xi_i R(x_i,x_j)\xi_j\ge0.
\end{align}
For any symmetric and positive semidefinite function $R:\RR^p\times\RR^p\to\RR$, there exists a
mean zero Gaussian random field $W$ with covariance function $R$ and with sample paths in $\RR^{\RR^p}$. In general, $W$ will not have continuous
paths. For a general random field $W$, we associate with $W$ its intrinsic pseudometric $\rho$ on $\RR^p$, given by
\begin{align}
  \rho(x,y) = \sqrt{E(W(x)-W(y))^2}.
\end{align}
If the metric space $(\RR^p,\rho)$ is separable, we say that $W$ is
separable. In this case, $\|W\|_\infty = \sup_{x\in
  D}|W(x)|$ with probability one, for any countable subset $D$ of $\RR^p$ 
which is dense under the pseudometric $\rho$. In particular, whenever the $\sigma$-algebra on the
space where $W$ is defined is complete, $\|W\|_\infty$ is measurable.

The following lemma describes some important properties of a class of Gaussian
fields particularly relevant to us. The result is well known, see for
example \cite{MR0185641}, and can be proven quite directly using the approximation results in \cite{MR936377}.

\begin{lemma}
\label{lemma:FGaussianWellDef}
Let $F$ be a cumulative distribution function on $\RR^p$. There exists a
$p$-dimensional separable mean zero Gaussian field $W$ which has covariance function
$R:\RR^p\times\RR^p\to\RR$ given by $R(x,y)=F(x\land y)-F(x)F(y)$ for
$x,y\in\RR^p$, where $x\land y$ is the coordinate wise minimum of $x$ and
$y$. With $\QQ$ denoting the rationals, it holds that $\|W\|_\infty=\sup_{x\in\QQ^p}|W(x)|$ and $\|W\|_\infty$ is almost surely finite.
\end{lemma}

For a fixed cumulative distribution function $F$, we refer to the Gaussian
field described in Lemma \ref{lemma:FGaussianWellDef} as an $F$-Gaussian
field. We are now ready to formulate our results on asymptotic scenarios in
ICA models.

Theorem \ref{theorem:ClassicalAsymptoticIdentifiability} describes the classical
asymptotic scenario, where the error distribution does not depend on the
sample size $n$. Fix a nondegenerate mean zero probability distribution $\zeta$ on
$(\RR,\BBB)$ and a matrix $A\in\MM(p,p)$. As
in the previous section, we let $F^A$ denote the cumulative distribution
function of $L_A(\zeta^{\otimes p})$, corresponding to the distribution of
$A\epsilon$ when $\epsilon$ is a $p$-dimensional variable with independent
coordinates having distribution $\zeta$. Consider a probability space
$(\Omega,\FFF,P)$ endowed with independent variables $(X_k)_{k\ge 1}$ with
cumulative distribution function $F^A$. Let $\FF_n^A$ be the empirical distribution
function of $X_1,\ldots,X_n$. Also assume that we are given an $F^A$-Gaussian field $W$ on
$(\Omega,\FFF,P)$.

\begin{theorem}
\label{theorem:ClassicalAsymptoticIdentifiability}
Let $c\ge 0$ be a continuity point of the distribution of
$\|W\|_{\infty}$. Then
\begin{align}
  \lim_{n\to\infty}P( \sqrt{n} \| \FF_n^A -F^A \|_\infty > c) & = P(\|W\|_\infty > c),\label{eq:ClassicalIdentifiable}
\end{align}
while in the case where $F^A \neq F^B$, it holds that
\begin{align}
  \lim_{n\to\infty}P( \sqrt{n} \|\FF_n^A-F^B\|_\infty > c) &=1.\label{eq:ClassicalDistinguish}
\end{align}
\end{theorem}
Equations (\ref{eq:ClassicalIdentifiable}) and
(\ref{eq:ClassicalDistinguish}) roughly state that in the classical asymptotic scenario, $\sqrt{n}\|\FF_{n}^A - F^A\|_{\infty}$ converges in
distribution to $\|W\|_{\infty}$, while $\sqrt{n} \|\FF_n^A -
F^B\|_{\infty}$ is not bounded in probability if $F^A \neq
F^B$. Note that Lemma \ref{lemma:GaussianUniqueness} gives us conditions
for $F^A=F^B$ and $F^A \neq F^B$ depending on $\zeta$.

Next, we consider an asymptotic scenario where the error distribution
is contaminated and the amount of contamination depends on the sample size $n$. As in Section
\ref{sec:UpperBound}, $\xi$ and $\zeta$ are fixed nondegenerate mean zero probability measures on
$(\RR,\BBB)$ with $\xi\neq\zeta$,
$P_e(\beta)=\beta\xi+(1-\beta)\zeta$, $A\in\MM(p,p)$ is a fixed matrix, $F^A$ is the cumulative
distribution function of $L_A(\zeta^{\otimes p})$ and $F^A_\beta$ is the
cumulative distribution function of $L_A(P_e(\beta)^{\otimes
  p})$. Thus, $F_\beta^A$ is the cumulative distribution function of
$A\eps$, where $\eps$ is a $p$-dimensional variable with
independent coordinates having distribution $P_e(\beta)$. Consider a sequence $(\beta_n)$
in $(0,1)$, and consider a probability space $(\Omega,\FFF,P)$ endowed with a triangular array
$(X_{nk})_{1\le k\le n}$ such that for each $n$, the variables
$X_{n1},\ldots,X_{nn}$ are independent variables with cumulative
distribution function $F_{\beta_n}^A$.  Let $\FF_{\beta_n}^A$ be the
empirical distribution function of $X_{n1},\ldots,X_{nn}$. Also
assume that we are given an $F^A$-Gaussian field $W$ on
$(\Omega,\FFF,P)$. We are interested in the asymptotic properties of
$\FF^A_{\beta_n}$. Theorem \ref{theorem:AsymptoticNonIdentifiability} is
our main result for this type of asymptotic scenarios.

\begin{theorem}
\label{theorem:AsymptoticNonIdentifiability}
Let $\lim_n \sqrt{n}\beta_n=k$ for some $k\ge0$. If
$F^A=F^B$, then
\begin{align}
  P(\|W\|_\infty > c + 4pk\|\xi-\zeta\|_\infty)
  &\le \liminf_{n\to\infty}P( \sqrt{n} \|\FF_{\beta_n}^A - F_{\beta_n}^B \|_\infty > c)\notag\\
  &\le \limsup_{n\to\infty}P( \sqrt{n} \|\FF_{\beta_n}^A - F_{\beta_n}^B \|_\infty > c)\notag\\
  &\le P(\|W\|_\infty \ge c - 4pk\|\xi-\zeta\|_\infty).\label{eq:NonstandardNoIdentify}
\end{align}
In particular, if $k=0$ and $c$ is a continuity point of
the distribution of $\|W\|_\infty$, we have
\begin{align}
\label{eq:SimpleAsympNonIdentifyLimit}
    \lim_{n\to\infty}P( \sqrt{n} \|\FF_{\beta_n}^A - F_{\beta_n}^B \|_\infty > c)
    =P(\|W\|_\infty > c).
\end{align}
\end{theorem}

Theorem \ref{theorem:AsymptoticNonIdentifiability} essentially shows that for the
asymptotic scenario considered, the convergence of $F^A_{\beta_n}$ to $F^A$
is fast enough to ensure that the asymptotic properties of
$\FF_{\beta_n}^A$ are determined by $F^A$ instead of $F^A_{\beta_n}$. Corollary
\ref{coro:GaussianAsymptoticNonIdentifiability} applies this result to the
case where the error distributions become close to Gaussian without being Gaussian.

\begin{corollary}
\label{coro:GaussianAsymptoticNonIdentifiability}
Assume that $\lim_n \sqrt{n}\beta_n=0$. Let $A,B\in\MM(p,p)$ be invertible. Assume that $AA^t=BB^t$
while $A\neq B\Lambda P$ for all diagonal $\Lambda$ with $\Lambda^2=I$ and
all permutation matrices $P$. Let $\zeta$ be a nondegenerate Gaussian
distribution and let $\xi$ be such that $P_e(\beta)$ is non-Gaussian for
all $\beta\in(0,1)$. Let $c$ be a point of continuity for the
distribution of $\|W\|_\infty$, with $W$ an $F^A$-Gaussian field. Then
\begin{enumerate}
\item $F_{\beta_n}^A\neq F_{\beta_n}^B$ for all $n\ge1$.
\item $\lim_{n\to\infty}P( \sqrt{n} \|\FF_{\beta_n}^A - F_{\beta_n}^B
  \|_\infty > c)= P(\|W\|_\infty > c)$.
\end{enumerate}
\end{corollary}

Statement $(1)$ of Corollary
\ref{coro:GaussianAsymptoticNonIdentifiability} shows that for any finite
$n$, we are in the case where, were the error distribution not changing
with $n$, it would be possible to asymptotically distinguish
$F_{\beta_n}^A$ and $F_{\beta_n}^B$ at rate $1/\sqrt{n}$ as in (\ref{eq:ClassicalDistinguish}) of the classical
case. However, statement $(2)$ shows that as $n$ increases and the error
distribution becomes closer to a Gaussian distribution, distinguishing
$F_{\beta_n}^A$ and $F_{\beta_n}^B$ at rate $1/\sqrt{n}$ is nonetheless impossible, with a limit
result similar to (\ref{eq:ClassicalIdentifiable}). Note that the condition in Corollary \ref{coro:GaussianAsymptoticNonIdentifiability} involving $A\neq B\Lambda P$
is the minimum requirement for non-Gaussian error distributions to asymptotically distinguish $F^A$ and $F^B$ in
the classical scenario (see Lemma 3.5).


Theorem \ref{theorem:AsymptoticNonIdentifiability} and Corollary
\ref{coro:GaussianAsymptoticNonIdentifiability} cover the case $\beta_n = o(1/\sqrt{n})$, in particular
the case $\beta_n = n^{-\rho}$ for $\rho> 1/2$. We end this section with a
result showing that, under some further regularity conditions,
distinguishing $F^A_{\beta_n}$ and $F^B_{\beta_n}$ at rate
$1/\sqrt{n}$ is possible when $0<\rho<1/2$.

\begin{theorem}
\label{theorem:AsymptoticIdentifiabilitySlowBeta}
Let $\rho\in(0,1/2)$ and let $\beta_n = n^{-\rho}$. For all
$A\in\MM(p,p)$, define
\begin{align}
\label{eq:Gamma1Special}
  \Gamma_1(A) =&\sum_{k=1}^p L_{A}\left(\zeta^{\otimes(k-1)}\otimes\frac{\xi-\zeta}{\|\xi-\zeta\|_\infty}\otimes\zeta^{\otimes(p-k)}\right).
\end{align}
If either $F^A\neq F^B$ or $F^A=F^B$ and $\Gamma_1(A)\neq\Gamma_1(B)$, then
\begin{align}
  \lim_{n\to\infty}P( \sqrt{n} \|\FF_{\beta_n}^A - F_{\beta_n}^B \|_\infty > c)=1.
\end{align}
\end{theorem}

As can be seen from the proof
of Theorem \ref{theorem:AsymptoticIdentifiabilitySlowBeta}, the
measure $L_A(P_e(\beta)^{\otimes p})$ can be written as a polynomial
of degree $p$ in $\beta$, where the constant term corresponds to $F^A$
and the first order term corresponds to $\Gamma_1(A)$, and similarly for
$L_B(P_e(\beta)^{\otimes p})$. In this light, Theorem
\ref{theorem:AsymptoticIdentifiabilitySlowBeta} shows that in the
absence of a difference between the constant terms of
$L_B(P_e(\beta)^{\otimes p})$ and $L_A(P_e(\beta)^{\otimes p})$, having
different first order terms is a sufficient criterion for distinguishing $F^A_{\beta_n}$
and $F^B_{\beta_n}$ at rate $1/\sqrt{n}$.

\section{Numerical experiments}

\label{sec:ICANumerical}

In this section, we carry out numerical experiments related
to the results in Section \ref{sec:Identifiability}. To make our experiments feasible, we
consider the scenario where $p=2$, $\zeta$ is the standard normal
distribution and $\xi$ is the standard exponential distribution. We consider the two matrices
\begin{align}
  A =\left[\begin{array}{cc}
      1 & 0 \\ \alpha & \sqrt{1-\alpha^2}\end{array}\right]
  \quad\textrm{ and }\quad
  B =\left[\begin{array}{cc}
      \sqrt{1-\alpha^2} & \alpha \\ 0 & 1 \end{array}\right],\notag
\end{align}
where we set $\alpha=0.4$. These two matrices are related to Example \ref{example:ICACausal}. Note that $AA^t=BB^t$ while
$A\neq B\Lambda P$ for all diagonal $\Lambda$ with $\Lambda^2=I$ and
all permutation matrices $P$. These properties makes $A$ and $B$ appropriate for evaluating the results of Section \ref{sec:Identifiability}. 
Fix $\beta\in(0,1)$ and let $\eps$ be a two-dimensional random variable with independent marginales and marginal
distributions equal to $P_e(\beta)=\beta\xi+(1-\beta)\NNN$, where
$\NNN$ denotes the standard normal distribution. The benefit of
this setup is that the cumulative distribution
functions $F^A_\beta$ and $F^B_\beta$ of $A\eps$ and $B\eps$ can be calculated in
semi-analytical form, depending only on
elementary functions and cumulative distribution functions for
one-dimensional and two-dimensional normal distributions.

We consider numerical evaluation of the sequence in the left-hand side of (\ref{eq:SimpleAsympNonIdentifyLimit}) and approximation of
its limit for $\beta_n = n^{-\rho}$ with varying $\rho>0$. We use a Monte Carlo approximation to evaluate
the probability. To be concrete, for fixed $n$ and $\beta_n = n^{-\rho}$, we make the approximation
\begin{align}
\label{eq:MCEstimate}
  P( \sqrt{n} \|\FF_{\beta_n}^A - F_{\beta_n}^B \|_\infty > c)  
  &\approx\frac{1}{N}\sum_{k=1}^N 1_{(X_k > c)}
\end{align}
for some fixed $N$, where $(X_k)$ are independent and identically distributed variables with the same distribution as
$\sqrt{n} \|\FF_{\beta_n}^A - F_{\beta_n}^B \|_\infty$. In order to simulate values from $X_k$, we first
simulate $n$ variables with distribution $A\epsilon$, this allows us to calculate the empirical cumulative
distribution function in any point. Due to properties of empirical cumulative distribution functions,
the supremum in $\|\FF_{\beta_n}^A - F_{\beta_n}^B \|_\infty$ can be reduced to a finite one. However, for the purpose
of reducing computational time, we are forced to approximate the finite maximum with a maximum over some fewer
points. This means that our probability estimates in general will be biased downwards. Also because of time
considerations, we restrict ourselves to using $N = 1000$ in (\ref{eq:MCEstimate}).

Coupling the above with the semi-analytical form of the exact cumulative distribution function, we are capable of 
simulating values of $X_k$ and evaluating the Monte Carlo estimate (\ref{eq:MCEstimate}). Our numerical results are
summarized in Figure \ref{fig:supExcesswithn}.

\begin{figure}[htb]
  \begin{center}
    \includegraphics[angle=0, scale=0.33]{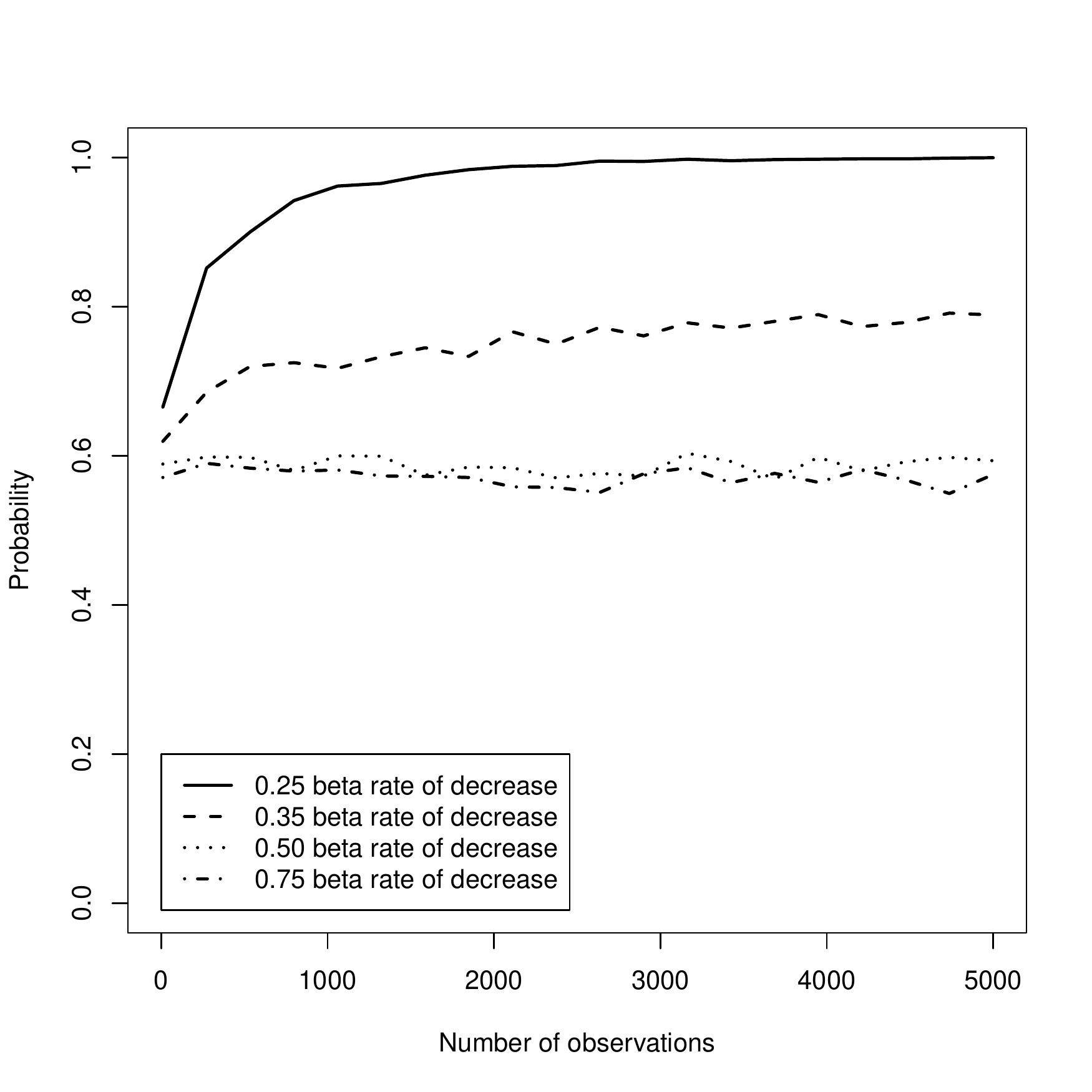}
    \includegraphics[angle=0, scale=0.33]{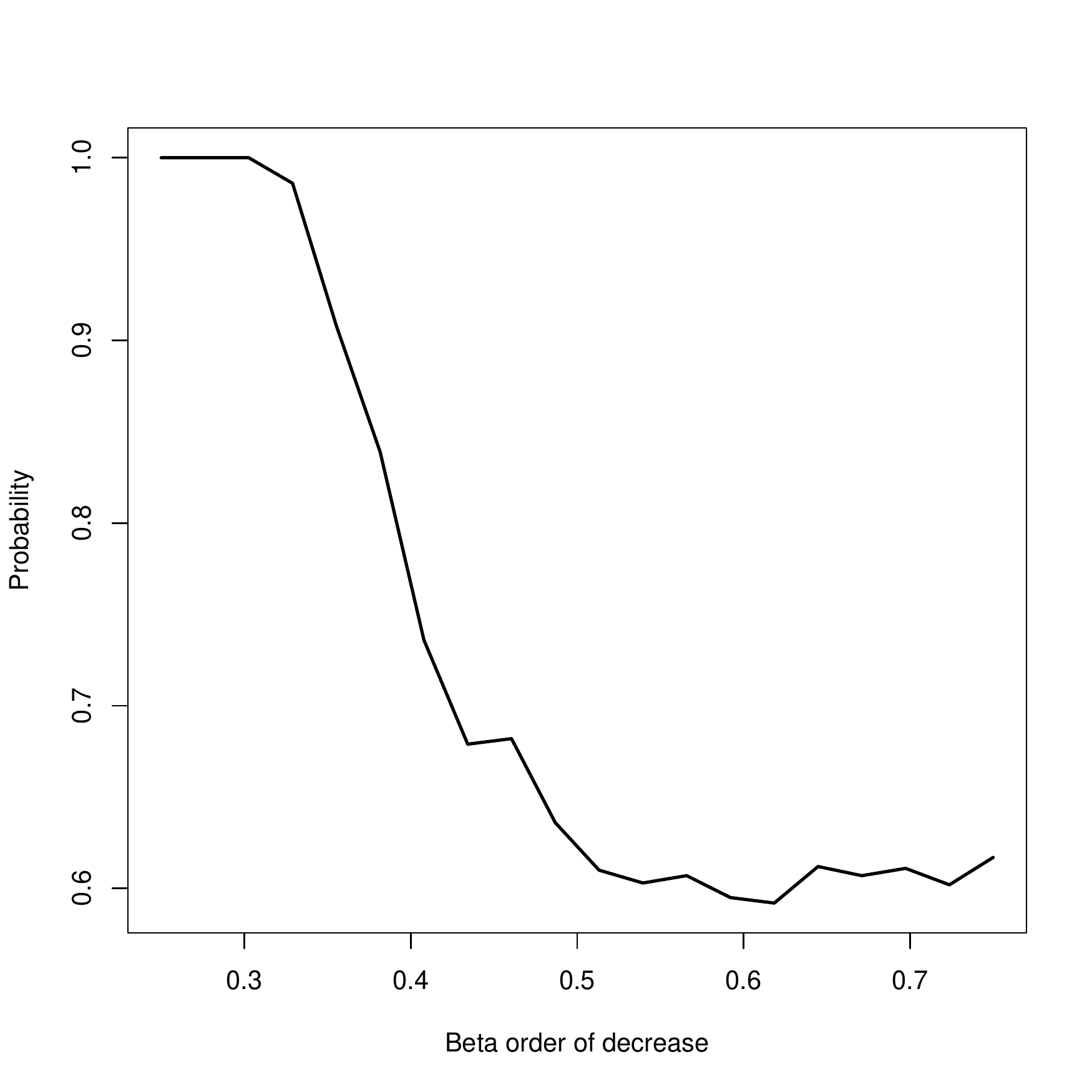}
  \end{center}
\caption{Left: Estimates of $P( \sqrt{n} \|\FF_{\beta_n}^A - F_{\beta_n}^B
  \|_\infty > c)$ for $\beta_n = n^{-\rho}$ with $\rho$ equal to
  $0.25$, $0.35$, $0.50$ and $0.75$, with $n$ varying up to
  $5000$. Right: Estimates of $P( \sqrt{n} \|\FF_{\beta_n}^A - F_{\beta_n}^B
    \|_\infty > c)$ for $\beta_n = n^{-\rho}$ with $\rho$ varying from
    $0.25$ to $0.75$. Here, $n=50000$. In both cases, $c=1$.}
\label{fig:supExcesswithn}
\end{figure}

Theorem \ref{theorem:AsymptoticNonIdentifiability} states that for $\rho>1/2$, the left hand
side of (\ref{eq:MCEstimate}) has a limit less than one. For $\rho$ equal to 0.50 and 0.75, our numerical
results in Figure \ref{fig:supExcesswithn} are in agreement with this.

On the other hand, for $\rho<1/2$, Theorem \ref{theorem:AsymptoticIdentifiabilitySlowBeta} states that the
limit of the left hand side of (\ref{eq:MCEstimate}) equals 1. The results in Figure \ref{fig:supExcesswithn}
indeed confirm this for $\rho=0.25$. For $\rho$ between $0.3$ and $0.5$, however, the
results are less clear. We think this is caused by the fact that the sample size required to see the asymptotic
behavior in Theorem \ref{theorem:AsymptoticIdentifiabilitySlowBeta} increases as $\rho$ tends closer to $1/2$.
To see this, note that for large $n$, we have for some constant $K>0$ that
\begin{align}
    P(\sqrt{n}\|\FF^A_{\beta_n}-F^B_{\beta_n}\|_\infty>c)
  &\approx P(\sqrt{n}\|F^A_{\beta_n}-F^B_{\beta_n}\|_\infty>c)
   \approx 1_{(\sqrt{n}K\beta_n>c)},
\end{align}
since $||\mathbb F_{\beta_n}^A - F_{\beta_n}^B||_{\infty} \approx || F_{\beta_n}^A - F_{\beta_n}^B ||_\infty$ and
$||F_{\beta_n}^A - F_{\beta_n}^B ||_{\infty}$ is approximately linear in $\beta_n$ by Corollary 3.4. Therefore,
$P(\sqrt{n}\|\FF^A_{\beta_n}-F^B_{\beta_n}\|_\infty>c)\approx 1$ if it holds that $\sqrt{n}\beta_n>c/K$, corresponding to
$n > \exp((\log c/K)/(1/2-\rho))$. This indicates that the $n$ required to detect the limiting value grows
exponentially fast in $(1/2-\rho)^{-1}$.

\section{Discussion}

\label{sec:Discussion}



We studied ICA models for error distributions which
have independent and identically distributed coordinates following contaminated distributions. Our main
theoretical contribution is Theorem \ref{theorem:AsymptoticNonIdentifiability}, which shows that
for $\beta_n = n^{-\rho}$ and $\rho\ge 1/2$, $\FF_{\beta_n}^A$ is asymptotically as close to $F_{\beta_n}^B$ as to
$F_{\beta_n}^A$ in the uniform norm whenever $F^A = F^B$. For contaminated Gaussian distributions, the requirement $F^A = F^B$
corresponds to $AA^t = BB^t$. Our results thus indicate that consistent estimation of $A$ up
to sign reversion and permutation of columns is not possible in this asymptotic scenario. In
particular, causal inference as described in Example \ref{example:ICACausal} (using LiNGAM) would suffer in such scenarios.

The proof of our main theoretical result, Theorem \ref{theorem:AsymptoticNonIdentifiability}, rests on two partial results:
\begin{enumerate}
\item Lemma \ref{lemma:AsymptoticConvWaart}, stating that when $F_n$ is a sequence of cumulative distribution functions
  converging uniformly to $F$, and $\FF_n$ is an empirical process based on $n$ independent
  observations of variables with cumulative distribution function
  $F_n$, then $\sqrt{n}(\FF_n-F_n)$ converges weakly in $\ell_\infty(\RR^p)$.
\item Theorem \ref{theorem:GrossErrorAsymptotics}, which is used to obtain that the convergence of the distribution
  functions $F^A_\beta$ of is asymptotically linear in $\beta$ as $\beta$ tends to zero. This result,
 combined with Lemma \ref{lemma:KNormBound}, allows us to obtain an asymptotic bound on
 $\|F^A_\beta-F^B_\beta\|_\infty$ in Corollary \ref{corr:TotalNormBound}.
\end{enumerate}

In Theorem \ref{theorem:AsymptoticIdentifiabilitySlowBeta}, we also
considered the case of slower rates of decrease in the level of
contamination, namely rates $n^{-\rho}$ for $0<\rho<1/2$. Our results
here indicate that in such asymptotic scenarios, identifiability of
the mixing matrix up to sign reversion and permutation of columns is possible,
subject to some regularity conditions related to the
$\Gamma_1$ signed measures of (\ref{eq:Gamma1Special}).

Our results are asymptotic in nature, considering limiting scenarios for a sequence of noise distributions. Theory on ICA applying such sequences
is not common. One paper using similar methodologies is \cite{YuanSamworth2012}. The authors of that paper are mainly interested
in estimation of $A^{-1}$ using nonparametric maximum likelihood. For this purpose, they introduce
the log-concave ICA projection. In Theorem 5 of their paper, they essentially prove that the set
of log-concave ICA projections for which the corresponding ICA model is identifiable is an open set,
using sequences of ICA models (parametrized through sequences of probability measures). We hope that
the present results as well as those of \cite{YuanSamworth2012} indicate that considering the
properties of ICA for various limiting scenarios of noise distributions may be an area for results
on ICA not yet fully reaped.

Our results also open up new research questions, such as the
following: Is it possible to characterize the matrices $A$ and $B$
such that the regularity condition $\Gamma_1(A)\neq \Gamma_2(B)$ of
Theorem \ref{theorem:AsymptoticIdentifiabilitySlowBeta} holds? Also,
together, Theorem \ref{theorem:AsymptoticNonIdentifiability} and
Theorem \ref{theorem:AsymptoticIdentifiabilitySlowBeta} describe the
behaviour of the empirical process $\FF^A_{\beta_n}$ for asymptotic
scenarios of the form $\beta_n=n^{-\rho}$ for $\rho>0$, in particular
describing the difficulty of using $\FF^A_{\beta_n}$ to distinguish
$F_{\beta_n}^A$ and $F^B_{\beta_n}$. Is it possible to obtain
finite-sample bounds instead of limiting behaviours in these results?
How do Theorem \ref{theorem:AsymptoticNonIdentifiability} and Theorem
\ref{theorem:AsymptoticIdentifiabilitySlowBeta} translate into results
on the ability of practical algorithms such as the fastICA algorithm,
see \cite{AH1999}, to distinguish the correct mixing matrix? Is it
possible to use similar techniques to analyze identifiability of the
mixing matrix in asymptotic scenarios where the number of components $p$ tends to infinity? Do
the present results extend to cases where the coordinates of the error
distributions are not contaminated normal distributions, or when the
coordinates are not identically distributed?

Our results have been motivated by applications in causal inference. Besides linear SEMs with non-Gaussian noise as discussed in Example \ref{example:ICACausal}, there are other
settings where the underlying causal structure is completely
identifiable, such as non-linear SEMs with almost arbitrary additive
noise and linear SEMs with additive Gaussian noise of
equal variance, see e.g. \cite{HJM} and \cite{PB}, respectively. We may also ask
whether one can use similar techniques to those presented here in order to study identifiability in these models when the
structural equations are close to linear or the variance of the errors are close to equal,
respectively. Also, we may consider the asymptotic behaviour of the model when we, instead
of considering perturbations of the error distribution, consider
perturbations of the additive nature of the noise, such that our
SEM is defined by $X = f(X,\eps)$, where $f(x)\approx Ax + \eps$, corresponding to models
with near-additive noise, and we let $f$ tend to a function linear in $\eps$?

In light of these open questions, our present results should be
seen as a small step towards a better understanding of the identifiability
of the mixing matrix for ICA for error distributions which are close
to Gaussian but not Gaussian. We hope that this paper will lead to
more work in this direction.

\appendix
\section{Proofs}

\label{sec:Proofs}

\subsection{Proofs for Section \ref{sec:UpperBound}} \textit{Proof of Theorem \ref{theorem:GrossErrorAsymptotics}.}
First note that we have $P_e(\beta)-\zeta=\beta(\xi-\zeta)$. Taking norms, this
implies $\|P_e(\beta)-\zeta\|_\infty=\beta\|\xi-\zeta\|_\infty$ and
\begin{align}
  \frac{P_e(\beta)-\zeta}{\|P_e(\beta)-\zeta\|_\infty}
  =\frac{\xi-\zeta}{\|\xi-\zeta\|_\infty} = \nu.
\end{align}
We then also have $P_e(\beta) = \zeta + \beta\|\xi-\zeta\|_{\infty} \nu$.

We now analyze $P_e(\beta)^{\otimes p}$. For Borel subsets $C_1,\ldots,C_p$ of $\RR$, we have
\begin{align}
  P_e(\beta)^{\otimes p}(C_1\times\cdots\times C_p)
  &=(\zeta+\beta\|\xi-\zeta\|_\infty\nu)^{\otimes p}(C_1\times\cdots\times C_p)\notag\\
  &=\prod_{k=1}^p (\zeta(C_k)+\beta\|\xi-\zeta\|_\infty\nu(C_k))\notag\\
  &=\sum_{k=0}^p \beta^k\|\xi-\zeta\|_\infty^k \sum_{\alpha\in S_k}\prod_{i=1}^p\zeta(C_i)^{1-\alpha_i}\nu(C_i)^{\alpha_i},
\end{align}
where $S_k =\left\{\alpha\in\{0,1\}^p\left| \sum_{i=1}^p
    \alpha_i=k\right.\right\}$, and the last equality follows
since
\begin{align}  \prod_{k=1}^p(a_k+\gamma b_k)=\sum_{k=0}^p\gamma^k\sum_{\alpha\in
    S_k}\prod_{i=1}^p a_i^{1-\alpha_i}b_i^{\alpha_i}, \qquad \text{for}\, a,b\in\RR^p\,\text{and}\,\gamma\in \RR.
\end{align}
Defining $\mu_0=\zeta$ and $\mu_1=\nu$, we then obtain
\begin{align}
  P_e(\beta)^{\otimes p}(C_1\times\cdots\times C_p) =\sum_{k=0}^p \beta^k\|\xi-\zeta\|_\infty^k \sum_{\alpha\in S_k}(\otimes_{i=1}^p\mu_{\alpha_i})(C_1\times\cdots\times C_p).
\end{align}
Letting $\Gamma_k = \sum_{\alpha\in S_k}L_A(\otimes_{i=1}^p\mu_{\alpha_i})$, this yields
\begin{align}
\label{eq:LAPeBetaGammaExp}
  L_{A}(P_e(\beta)^{\otimes p})=\sum_{k=0}^p \beta^k\|\xi-\zeta\|_\infty^k \Gamma_k.
\end{align}
Next, note that $\Gamma_0 = L_{A}(\zeta^{\otimes p})$, so that
\begin{align}
  \lim_{\beta\to0}
  \frac{L_A(P_e(\beta)^{\otimes p})-L_A(\zeta^{\otimes p})}{\|P_e(\beta)-\zeta\|_\infty}
  &=\lim_{\beta\to0}\sum_{k=1}^p \beta^{k-1}\|\xi-\zeta\|_\infty^{k-1} \Gamma_k\notag\\
  &=\Gamma_1
  =\sum_{k=1}^p L_{A}(\zeta^{\otimes(k-1)}\otimes\nu\otimes\zeta^{\otimes(p-k)}).
\end{align}
In particular, this shows that for any $\eta>0$,
\begin{align}
   \limsup_{\beta\to0}\|F^A_\beta - F^A\|_\infty
  = &\limsup_{\beta\to0}\|L_A(P_e(\beta)^{\otimes p})-L_A(\zeta^{\otimes p})\|_\infty\notag\\
  \le &\limsup_{\beta\to0}(1+\eta)\|\Gamma_1\|_\infty\|P_e(\beta)-\zeta\|_\infty\notag\\  
  \le &\limsup_{\beta\to0}(1+\eta)\|\Gamma_1\|_\infty\beta \|\xi-\zeta\|_\infty=0,
\end{align}
so $F^A_\beta$ converges uniformly to $F^A$ as $\beta$ tends to zero.
\hfill$\Box$

\textit{Proof of Lemma \ref{lemma:OnlyContaminatedProps}.}
Let $\beta\in(0,1)$ and let $\alpha$ be such that $\|Q(\beta)-\zeta\|_\infty=\alpha\beta$. Let
$\xi=Q(1)$. We then have
\begin{align}
  \frac{Q(\beta)-\zeta}{\|Q(\beta)-\zeta\|_\infty}
   =\frac{Q(\beta)-\zeta}{\alpha\beta},\label{eq:TauBeta1}
\end{align}
while
\begin{align}
  \frac{Q(1)-\zeta}{\|Q(1)-\zeta\|_\infty}
   =\frac{\xi-\zeta}{\alpha}.\label{eq:TauBeta2}
\end{align}
By our assumptions, the right-hand sides in (\ref{eq:TauBeta1}) and (\ref{eq:TauBeta2}) are
equal. This implies $Q(\beta)=\beta\xi+(1-\beta)\zeta$.
\hfill$\Box$

To prove Lemma \ref{lemma:KNormBound}, we
first present a lemma relating the uniform norm of certain measures on
$(\RR^p,\BBB_p)$ to the uniform and total variation norms of some measures on $(\RR,\BBB)$.

\begin{lemma}
\label{lemma:ktvkBound}
Let $\mu_1,\ldots,\mu_p$ be signed measures on $(\RR,\BBB)$, and let $A \in \MM(p,p)$. Then for any $i\in \{1,\dots,p\}$, it holds that
\begin{align}
  \|L_A(\mu_1\otimes\cdots\otimes\mu_p)\|_\infty
  \le 2\|\mu_i\|_\infty\prod_{k\neq i}^p\|\mu_k\|_{tv}.
\end{align}
\end{lemma}
\begin{proof}
For any permutation
$\pi:\{1,\ldots,p\}\to\{1,\ldots,p\}$ and corresponding permutation matrix $P$, we have
$L_A(\mu_1\otimes\cdots\otimes\mu_p)=L_{AP^{-1}}(\mu_{\pi(1)}\otimes\cdots\otimes\mu_{\pi(p)})$. Hence,
it suffices to consider the case where $i=p$. Let $x\in\RR^p$ and define 
$I_x=(-\infty,x_1]\times\cdots\times(-\infty,x_p]$. Then Fubini's
theorem for signed measures yields
\begin{align}
  & \left|L_A(\mu_1\otimes\cdots\otimes\mu_p)(I_x)\right|\notag\\
    =&  \left|\int\cdots\int 1_{I_x}(L_A(y))\dv\mu_p(y_p)\cdots\dv\mu_1(y_1)\right|\notag\\
   \le& \int\cdots\int \left|\int 1_{I_x}(L_A(y))\dv\mu_p(y_p)\right|\dv|\mu_{p-1}|(y_{p-1})\cdots\dv|\mu_1|(y_1),\label{eq:KTvBound1}
\end{align}
where we have also used the triangle inequality for integrals with respect
to signed measures, which follows for example from Theorem 6.12 of
\cite{MR924157}. We now analyze the innermost integral of \eqref{eq:KTvBound1}. For fixed $y_1,\ldots,y_{p-1}$, we have
\begin{align}
  &\{y_p\in\RR\mid 1_{I_x}(L_A(y))=1\}\notag\\
   =& \{y_p\in\RR\mid \foralls i\le p: (Ay)_i \le x_i\}\notag\\
   =& \cap_{i=1}^p \{y_p\in\RR\mid a_{ip}y_p\le x_i-(a_{i1}y_1+\cdots+a_{i(p-1)}y_{p-1})\},
\end{align}
where $a_{ij}$ is the $(i,j)$'th entry of $A$. Hence, $\{y_p\in\RR\mid 1_{I_x}(L_A(y))=1\}$ is a finite intersection of
intervals, and is therefore itself an interval. This yields 
\begin{align}\label{eq: yp}
   |\mu_p(\{y_p\in\RR\mid 1_{I_x}(L_A(y))=1\})|\le2\|\mu_p\|_\infty.
\end{align} 
This inequality is immediate when the
interval is of the form $(-\infty,a]$ for some $a\in\RR$. If the interval is of the form $[a,\infty)$, we have
\begin{align}
  |\mu_p([a,\infty))|
  \le& |\mu_p(\RR)|+|\mu_p(-\infty,a)|\notag\\
  =&\lim_{b\to\infty}|\mu_p((-\infty,b])|+|\mu_p((-\infty,a-1/b])|
  \le 2\|\mu_p\|_\infty,
\end{align}
and similarly for other types of intervals, whether bounded or unbounded,
open, half-open or closed. Combining \eqref{eq:KTvBound1} and \eqref{eq: yp} yields
\begin{align}
  & \left|L_A(\mu_1\otimes\cdots\otimes\mu_p)(I_x)\right|\notag\\
   \le& \int\cdots\int 2\|\mu_p\|_\infty\dv|\mu_{p-1}|(y_{p-1})\cdots\dv|\mu_1|(y_1)
  =2\|\mu_p\|_\infty\prod_{k=1}^{p-1}\|\mu_k\|_{tv}.
\end{align}
\end{proof}

\textit{Proof of Lemma \ref{lemma:KNormBound}.}
By Lemma \ref{lemma:ktvkBound}, we have
\begin{align}
  \|L_A(\zeta^{\otimes(k-1)}\otimes\nu\otimes\zeta^{\otimes(p-k)})\|_\infty
  &\le 2\|\nu\|_\infty\|\zeta\|_{tv}^{p-1}=2.
\end{align}
Applying the triangle inequality, we therefore obtain
\begin{align}
  \left\|\sum_{k=1}^p L_{A}\left(\zeta^{\otimes(k-1)}\otimes\nu\otimes\zeta^{\otimes(p-k)}\right)\right\|_\infty
  & \le2p.
\end{align}
\hfill$\Box$

\textit{Proof of Lemma \ref{lemma:GaussianUniqueness}.}
\textbf{Proof of (1)}. With $\zeta$ Gaussian with mean zero and variance $\sigma^2$,
$L_A(\zeta^{\otimes p})$ is Gaussian with mean zero and variance $\sigma^2AA^t$, and so the
result is immediate for this case.

\textbf{Proof of (3)}. Now consider the case
where $\zeta$ is not a symmetric distribution. As $L_P(\zeta^{\otimes
  p})=\zeta^{\otimes p}$ holds for any permutation matrix $P$, we obtain that
if $A=BP$, then $L_A(\zeta^{\otimes p})= L_B(\zeta^{\otimes
  p})$ and so $F^A=F^B$, proving one implication.

Conversely, assume that
$F^A=F^B$, meaning that $L_A(\zeta^{\otimes p})=
L_B(\zeta^{\otimes p})$. As $\zeta$ is nondegenerate and non-Gaussian and $A$
and $B$ are invertible, Theorem 4 of \cite{EK2004}
shows that $A=B\Lambda P$, where $\Lambda\in\MM(p,p)$ is an invertible
diagonal matrix and $P\in\MM(p,p)$ is a permutation matrix. This yields
\begin{align}
  \zeta^{\otimes p}
  =L_{B^{-1}}(L_B(\zeta^{\otimes p}))
  =L_{B^{-1}}(L_A(\zeta^{\otimes p}))
  =L_{\Lambda P}(\zeta^{\otimes p})
  =L_{\Lambda}(\zeta^{\otimes p}).
\end{align}
Now let $Z$ be a random variable with distribution $\zeta$. The above then
yields that for all $i$, $\Lambda_{ii}Z$ and $Z$ have the same
distribution. In particular, $|\Lambda_{ii}||Z|$ and $|Z|$ have the same
distribution, so $P(|Z|\le z/|\Lambda_{ii}|)=P(|Z|\le z)$ for all
$z\in\RR$. As $Z$ is not almost surely zero, there is $z\neq 0$ such that
$P(|Z|\le z-\eps)<P(|Z|\le z+\eps)$ for all $\eps>0$. This yields $|\Lambda_{ii}|=1$. Next,
let $\varphi$ denote the characteristic function of $Z$.  We then have
$\varphi(\Lambda_{ii}\theta)=\varphi(\theta)$ for all $\theta\in\RR$. As
$Z$ is not symmetric, there is a $\theta\in\RR$ such that
$\varphi(\theta)\neq\varphi(-\theta)$. Therefore, $\Lambda_{ii}=-1$ cannot
hold, so we must have $\Lambda_{ii}=1$. We conclude that $\Lambda$ is the identity matrix and
thus $A=BP$, as required.

\textbf{Proof of (2)}. Finally, consider a symmetric probability measure $\zeta$. It is then immediate that
when $\Lambda$ and $P$ are as in the statement of the lemma, it holds that
$L_{\Lambda P}(\zeta^{\otimes p})=\zeta^{\otimes p}$ and thus $F^A=F^B$
whenever $A=B\Lambda P$. The converse implication follows as in the proof
of (3).
\hfill$\Box$

\subsection{Proofs for Section \ref{sec:Identifiability}}

Before proving Theorem \ref{theorem:ClassicalAsymptoticIdentifiability} and
Theorem \ref{theorem:AsymptoticNonIdentifiability}, we show a result on
empirical processes. Recall that for a metric space $(M,d)$, the
$\eps$-covering number $N(\eps,M,d)$ is the minimum number of open balls of
radius $\eps$ which is required to cover $(M,d)$, see, e.g., Section 2.1.1 of \cite{MR1385671}.

\begin{lemma}
\label{lemma:TotallyBoundedIntrinsitc}
Fix a cumulative distribution function $F$. Define $\rho:\RR^p\times\RR^p$ by
\begin{align}
  \rho(x,y) = \sqrt{F(x)+F(y)-2F(x\land y)},
\end{align}
and let $I_x=(-\infty,x_1]\times\cdots\times(-\infty,x_p]$. Let $Z$ be
a variable with cumulative distribution function $F$. Then, the following holds:
\begin{enumerate}
\item $\rho$ is a pseudometric.
\item $\rho(x,y) = \sqrt{E(1_{I_x}(Z)-1_{I_y}(Z))^2}$.
\item $(\RR^p,\rho)$ is totally bounded.
\end{enumerate}
\end{lemma}
\begin{proof}
First note that
\begin{align}
  \rho(x,y)^2 =& F(x)+F(y)-2F(x\land y)\notag\\
              =& E1_{I_x}(Z)+E1_{I_y}(Z)-2E1_{I_x}(Z)1_{I_y}(Z)\notag\\
              =& E(1_{I_x}(Z)-1_{I_y}(Z))^2,
\end{align}
proving claim $(2)$. It is then immediate that $\rho$ is a pseudometric,
proving claim $(1)$. Next, it holds that
$(\RR^p,\rho)$ is totally bounded if and only if $N(\eps,\RR^p,\rho)$ is finite for
all positive $\eps$. Let $Q$ be the distribution corresponding to the cumulative distribution
function $F$, and let $\LLL^2(\RR^p,\BBB_p,Q)$ be the space of
Borel measurable functions from $\RR^p$ to $\RR$ which are
square-integrable with respect to $Q$. Let $\|\cdot\|_{2,Q}$ denote
the usual seminorm on $\LLL^2(\RR^p,\BBB_p,Q)$. Applying claim (2), it is immediate that
\begin{align}
   N(\eps,\RR^p,\rho)
  = N(\eps,(1_{I_x})_{x\in\RR^p},\|\cdot\|_{2,Q}).
\end{align}
Combining Example 2.6.1 and Exercise 2.6.9 of \cite{MR1385671}, we find that $(1_{I_x})_{x\in\RR^p}$ is a
Vapnik-Cervonenkis (VC) subgraph class with VC dimension
$p+1$. Furthermore, $(1_{I_x})_{x\in\RR^p}$ has envelope function constant and
equal to one. Therefore, Theorem 2.6.7 of \cite{MR1385671} shows that
$N(\eps,(1_{I_x})_{x\in\RR^p},\|\cdot\|_{2,Q})$ and thus $N(\eps,\RR^p,\rho)$
is finite, and so $(\RR^p,\rho)$ is totally bounded.
\end{proof}

\begin{lemma}
\label{lemma:AsymptoticConvWaart}
Let $(F_n)$ be a sequence of cumulative distribution functions on $\RR^p$,
and let $F$ be a cumulative distribution function on $\RR^p$. Let
$(X_{nk})_{1\le k\le n}$ be a triangular array such that for each $n$,
$X_{n1},\ldots,X_{nn}$ are independent with distribution $F_n$. Let $\FF_n$
be the empirical distribution function of $X_{n1},\ldots,X_{nn}$. If $F_n$ converges uniformly to $F$, then
$\sqrt{n}(\FF_n-F_n)$ converges weakly in $\ell_\infty(\RR^p)$ to an $F$-Gaussian field.
\end{lemma}

\begin{proof}
For $x,y\in\RR^p$ and $n\ge1$, let $R_n(x,y)=F_n(x\land y)-F_n(x)F_n(y)$
and also define $R(x,y)=F(x\land y)-F(x)F(y)$. Let $\rho$ be the
pseudometric of Lemma \ref{lemma:TotallyBoundedIntrinsitc} corresponding to
the cumulative distribution function $F$. Let
$Z_{nk}$ be the random field indexed by $\RR^p$ given by $Z_{nk}(x)=1_{I_x}(X_{nk})/\sqrt{n}$,
where we as usual put $I_x=(-\infty,x_1]\times\cdots\times(-\infty,x_p]$. We then have
\begin{align}
  \sum_{k=1}^n Z_{nk}(x)-EZ_{nk}(x)
  =&\frac{1}{\sqrt{n}}\sum_{k=1}^n 1_{I_x}(X_{nk})-F_n(x)\notag\\
  =&\sqrt{n}(\FF_n(x)-F_n(x)).\label{eq:ZToEmpF}
\end{align}
We will apply Theorem 2.11.1 of \cite{MR1385671} to prove that
$\sum_{k=1}^n Z_{nk}-EZ_{nk}$ and thus $\sqrt{n}(\FF_n-F_n)$ converges
weakly in $\ell_\infty(\RR^p)$. We may assume without loss of generality
that all variables are defined on a product probability space as described
in Section 2.11.1 of \cite{MR1385671}, and as the fields $(Z_{nk})$ can be
constructed using only countably many variables, the measurability
requirements in Theorem 2.11.1 of \cite{MR1385671} can be ensured. In order
to apply Theorem 2.11.1 of \cite{MR1385671}, first note that by Lemma
\ref{lemma:TotallyBoundedIntrinsitc}, $(\RR^p,\rho)$ is totally bounded and
so can be applied in Theorem 2.11.1 of \cite{MR1385671}. Also, the covariance function of
$\sum_{k=1}^n Z_{nk}-EZ_{nk}$ is
\begin{align}
  &\cov\left(\sum_{k=1}^n Z_{nk}(x)-EZ_{nk}(x),\sum_{k=1}^n Z_{nk}(y)-EZ_{nk}(y)\right)\notag\\
  =&\sum_{k=1}^n\sum_{i=1}^n EZ_{nk}(x)Z_{ni}(y)-EZ_{nk}(x)EZ_{ni}(y)\notag\\
  =&\frac{1}{n}\sum_{k=1}^n E1_{I_x}(X_{nk})1_{I_y}(X_{nk})-E1_{I_x}(X_{nk})E1_{I_y}(X_{nk})\notag\\
  =&F_n(x\land y)-F_n(x)F_n(y)=R_n(x,y).\label{eq:EmpCovCalc}
\end{align}
Note that
\begin{align}
  & |R(x,y)-R_n(x,y)|\notag\\
  \le& |F(x\land y)-F_n(x\land y)|+|F(x)F(y)-F_n(x)F_n(y)| \notag\\
  \le& |F(x\land y)-F_n(x\land y)|+|F(x)-F_n(x)|+|F_n(y)-F_n(y)|,
\end{align}
so as $F_n$ converges uniformly to $F$, $R_n$ converges uniformly to
$R$. Thus, the covariance functions of $\sum_{k=1}^n Z_{nk}-EZ_{nk}$
converge to $R$. Therefore, in order to apply Theorem 2.11.1 of
\cite{MR1385671}, it only remains to confirm that the conditions of
(2.11.2) in \cite{MR1385671} hold. Fixing $\eta>0$, we have
\begin{align}
  \sum_{k=1}^n E\|Z_{nk}\|_\infty^21_{(\|Z_{nk}\|_\infty>\eta)}
 =& \frac{1}{n}\sum_{k=1}^n E1_{I_x}(X_{nk})1_{(1_{I_x}(X_{nk})>\sqrt{n}\eta)}\notag\\
 \le& P(1_{I_x}(X_{n1})>\sqrt{n}\eta)\notag,
\end{align}
and so it is immediate that the first condition of (2.11.2) in
\cite{MR1385671} holds. Next, define $d_n^2(x,y)=\sum_{k=1}^n
(Z_{nk}(x)-Z_{nk}(y))^2$. We then also have for $x,y\in\RR^p$ that
\begin{align}
  d_n^2(x,y)
  =&\frac{1}{n}\sum_{k=1}^n (1_{I_x}(X_{nk})-1_{I_y}(X_{nk}))^2,\label{eq:dn2Char}
\end{align}
and therefore, $Ed_n(x,y)^2=F_n(x)+F_n(y)-2F_n(x\land y)$. Thus, $(x,y)\mapsto Ed_n(x,y)^2$ converges uniformly to $\rho^2$ on
$\RR^p\times\RR^p$. Therefore, we conclude that for any sequence $(\delta_n)$ of positive
numbers tending to zero, it holds for all $\eta>0$ that
\begin{align}
  \limsup_{n\to\infty}\sup_{x,y:\rho(x,y)\le\delta_n}Ed_n^2(x,y)
  \le& \limsup_{n\to\infty}\sup_{x,y:\rho(x,y)\le\delta_n}\rho(x,y)^2\notag\\
  \le& \limsup_{n\to\infty}\delta_n^2=0.
\end{align}
Hence, the second condition of
(2.11.2) in \cite{MR1385671} holds. In order to verify the final condition of
(2.11.2) in \cite{MR1385671}, first note that $d_n(x,y)^2=E_{\PP_n}(1_{I_x}-1_{I_y})^2$ by (\ref{eq:dn2Char}),
where $E_{\PP_n}$ denotes
integration with respect to $\PP_n$ and $\PP_n$ is the empirical measure
on $(\RR^p,\BBB_p)$ in $X_{n1},\ldots,X_{nn}$. Thus, $d_n(x,y)$ is the
$\LLL^2(\RR^p,\BBB_p,\PP_n)$ distance between the mappings $I_x$ and $I_y$,
and so
\begin{align}
  N(\eps,\RR^p,d_n)
  =N(\eps,(1_{I_x})_{x\in\RR^p},\|\cdot\|_{2,\PP_n})
  \le \sup_Q N(\eps,(1_{I_x})_{x\in\RR^p},\|\cdot\|_{2,Q}),
\end{align}
where $\|\cdot\|_{2,Q}$ denotes the norm on $\LLL^2(\RR^p,\BBB_p,Q)$ and the
supremum is over all probability measures $Q$ on $(\RR^p,\BBB_p)$. Thus,
the third condition of (2.11.2) in \cite{MR1385671} is satisfied if only it
holds that for all sequences $(\delta_n)$ of positive numbers tending to zero,
\begin{align}
\label{eq:ThirdSuffVaart}
  \lim_{n\to\infty}\int_0^{\delta_n} \sup_Q\sqrt{\log N(\eps,(1_{I_x})_{x\in\RR^p},\|\cdot\|_{2,Q})}\dv\eps=0.
\end{align}
However, Theorem 2.6.7 of \cite{MR1385671} yields a
constant $K>0$ such that for $0<\eps<1$,
\begin{align}
  \sup_Q N(\eps,(1_{I_x})_{x\in\RR^p},\|\cdot\|_{2,Q})
  \le K(p+1)(16e)^{p+1}\eps^{-2p}.
\end{align}
As a consequence, again for $0<\eps<1$,
\begin{align}
  \sup_Q \sqrt{\log N(\eps,(1_{I_x})_{x\in\RR^p},\|\cdot\|_{2,Q})}
  \le& \sqrt{\log K(p+1)(16e)^{p+1} - 2p\log \eps}.
\end{align}
By elementary calculations, we obtain for $0<c<d<1$ and $a,b>0$ that
\begin{align}
  &\int_c^d \sqrt{a-b\log x}\dv x\notag\\
  =&\left[ x\sqrt{a-b\log x}-\frac{e^{a/b}\sqrt{\pi b}}{2}\mathrm{erf}\left(\frac{\sqrt{a-b\log x}}{\sqrt{b}}\right)\right]_c^d,
\end{align}
where $\mathrm{erf}$ denotes the error function,
$\mathrm{erf}(x)=(2/\sqrt{\pi})\int_0^x \exp(-y^2)\dv y$. Therefore, we
conclude that for all $0<\eta<1$, the mapping $x\mapsto \sqrt{a-b\log x}$
is integrable over $[0,\eta]$. Thus, (\ref{eq:ThirdSuffVaart})
holds. Recalling (\ref{eq:ZToEmpF}), Theorem 2.11.1 of \cite{MR1385671} now shows that
$\sqrt{n}(\FF_n-F_n)$ converges weakly in $\ell_\infty(\RR^p)$. By
uniqueness of the finite-dimensional distributions of the limit, we find
that the limit is an $F$-Gaussian field.
\end{proof}

\textit{Proof of Theorem
  \ref{theorem:ClassicalAsymptoticIdentifiability}.} By Lemma
\ref{lemma:AsymptoticConvWaart} and the continuous mapping theorem, $\sqrt{n}\|\FF^A_n-F^A\|_\infty$ converges
weakly to $\|W\|_\infty$. Therefore, equation
(\ref{eq:ClassicalIdentifiable}) follows. In order to
prove equation \eqref{eq:ClassicalDistinguish}, consider $A$ and $B$ such
that $F^A\neq F^B$ and let $\|F^A-
F^B\|_{\infty} = \alpha$. Whenever
$\|\FF_{n}^A-F^A\|_\infty\le \alpha/2$, the reverse triangle
inequality yields
\begin{align}
  \|\FF_{n}^A-F^B\|_\infty
  =& \|\FF_{n}^A-F^A-(F^B-F^A)\|_\infty\notag\\
  \ge & |\|\FF_{n}^A-F^A\|_\infty-\|F^B-F^A\|_\infty|\notag\\
    = &|\|\FF_{n}^A-F^A\|_\infty-\alpha|
    \ge\alpha/2.
\end{align}
Since $\lim_{n\to \infty} P( \|\FF_{n}^A - F^A\|_{\infty} \le \alpha/2) =
1$ by Lemma \ref{lemma:AsymptoticConvWaart}, we obtain
\begin{align}
  &\limsup_{n\to\infty} P(  \sqrt{n}\|\FF_{n}^A-F^B\|_\infty \le c )\notag\\
  & \qquad = \limsup_{n\to\infty} P( \|\FF_{n}^A-F^B\|_\infty \le c/\sqrt{n}, \|\FF_{n}^A-F^A\|_\infty \le \alpha/2)\notag\\
  & \qquad \le \limsup_{n\to\infty} P( \|\FF_{n}^A-F^B\|_\infty \le c/\sqrt{n}, \|\FF_{n}^A-F^B\|_\infty \ge \alpha/2)=0.
\end{align}
Hence, $\lim_{n\to\infty} P(\sqrt{n}\| \FF_{n}^A-F^B\|_\infty \le c )=0$
and so (\ref{eq:ClassicalDistinguish}) holds.
\hfill$\Box$

\textit{Proof of Theorem \ref{theorem:AsymptoticNonIdentifiability}.}
By the  triangle inequality, we have the inequalities
\begin{align}
\label{eq:UpperLowerMainIneq}
     & P(\sqrt{n} \|\FF_{\beta_n}^A-F_{\beta_n}^A\|_\infty-\sqrt{n}\|F_{\beta_n}^B-F_{\beta_n}^A\|_\infty> c)\notag\\
  \le& P(\sqrt{n} \|\FF_{\beta_n}^A-F_{\beta_n}^B\|_\infty > c )\notag\\
  \le& P(\sqrt{n} \|\FF_{\beta_n}^A-F_{\beta_n}^A\|_\infty+\sqrt{n}\|F_{\beta_n}^B-F_{\beta_n}^A\|_\infty> c).
\end{align}
Let $\eta>0$. By Corollary \ref{corr:TotalNormBound}, we can choose $N\ge1$ such that for $n\ge N$,
\begin{align}
  \sqrt{n}\|F_{\beta_n}^B-F_{\beta_n}^A\|_\infty\le 4p(1+\eta)\sqrt{n}\beta_n\|\xi-\zeta\|_\infty.
\end{align}
By our assumptions, $\lim_n\sqrt{n}\beta_n=k$. Letting $\gamma>0$, we
then find for $n$ large that
\begin{align}
  \sqrt{n}\|F_{\beta_n}^B-F_{\beta_n}^A\|_\infty
  &\le 4p(1+\eta)(k+\gamma)\|\xi-\zeta\|_\infty.
\end{align}
For such $n$, the first inequality of (\ref{eq:UpperLowerMainIneq}) yields
\begin{align}
\label{eq:LowerLiminfHelper}
  &P(\sqrt{n}\|\FF_{\beta_n}^A-F_{\beta_n}^B\|_\infty > c)\notag\\
  \ge&P(\|\sqrt{n}(\FF_{\beta_n}^A-F_{\beta_n}^A)\|_\infty> c+\sqrt{n}\|F_{\beta_n}^B-F_{\beta_n}^A\|_\infty)\notag\\
  \ge&P(\|\sqrt{n}(\FF_{\beta_n}^A-F_{\beta_n}^A)\|_\infty> c+4p(1+\eta)(k+\gamma)\|\xi-\zeta\|_\infty).
\end{align}
Now recall from Theorem \ref{theorem:GrossErrorAsymptotics} that
$F_{\beta_n}^A$ converges uniformly to $F^A$. Therefore, Lemma
\ref{lemma:AsymptoticConvWaart} and the continuous mapping theorem show
that $\sqrt{n}\|\FF_{\beta_n}^A-F_{\beta_n}^A\|_\infty$ converges weakly to
$\|W\|_\infty$. As a consequence, (\ref{eq:LowerLiminfHelper}) yields
\begin{align}
  &\liminf_{n\to\infty}P(\sqrt{n}\|\FF_{\beta_n}^A-F_{\beta_n}^B\|_\infty > c)\notag\\
  \ge&P(\|W\|_\infty>c+4p(1+\eta)(k+\gamma)\|\xi-\zeta\|_\infty).
\end{align}
Letting $\eta$ and then $\gamma$ tend to zero, we obtain
\begin{align}
  \liminf_{n\to\infty}P(\sqrt{n}\|\FF_{\beta_n}^A-F_{\beta_n}^B\|_\infty > c)
  \ge& P(\|W\|_\infty>c+4p k\|\xi-\zeta\|_\infty).
\end{align}
Similarly, the second inequality of (\ref{eq:UpperLowerMainIneq}) yields
\begin{align}
\label{eq:UpperLimsupHelper}
  &P(\sqrt{n}\|\FF_{\beta_n}^A-F_{\beta_n}^B\|_\infty > c)\notag\\
  \le&P(\|\sqrt{n}(\FF_{\beta_n}^A-F_{\beta_n}^A)\|_\infty> c-\sqrt{n}\|F_{\beta_n}^B-F_{\beta_n}^A\|_\infty)\notag\\
  \le&P(\|\sqrt{n}(\FF_{\beta_n}^A-F_{\beta_n}^A)\|_\infty\ge c-4p(1+\eta)(k+\gamma)\|\xi-\zeta\|_\infty),
\end{align}
and by similar arguments as previously, we obtain
\begin{align}
  \limsup_{n\to\infty}P(\sqrt{n}\|\FF_{\beta_n}^A-F_{\beta_n}^B\|_\infty > c)
  \le& P(\|W\|_\infty\ge c-4p k\|\xi-\zeta\|_\infty).
\end{align}
Combining our results, we obtain (\ref{eq:NonstandardNoIdentify}).
\hfill$\Box$

\textit{Proof of Corollary
  \ref{coro:GaussianAsymptoticNonIdentifiability}.} As we have assumed that
$P_e(\beta_n)$ is non-Gaussian, it follows
from Lemma \ref{lemma:GaussianUniqueness} that $F^A_\beta\neq F^B_\beta$,
since $A\neq B\Lambda P$ for all diagonal $\Lambda$ with $\Lambda^2=I$ and
all permutation matrices $P$. This shows $(1)$. And as $AA^t=BB^t$ and
$\zeta$ is Gaussian, Lemma \ref{lemma:GaussianUniqueness} yields $F^A=F^B$,
so Theorem \ref{theorem:AsymptoticNonIdentifiability} yields $(2)$.
\hfill$\Box$

\textit{Proof of Theorem
  \ref{theorem:AsymptoticIdentifiabilitySlowBeta}.} Note that for any
$x\in\RR^p$, we have
\begin{align}
\label{eq:OnePointBound}
      & P( \sqrt{n} \|\FF_{\beta_n}^A - F_{\beta_n}^B \|_\infty > c)\notag\\
  \ge & P( \sqrt{n} |\FF_{\beta_n}^A(x) - F_{\beta_n}^B(x) | > c)\notag\\
  =   & P(  |\sqrt{n}(\FF_{\beta_n}^A(x)-F^A_{\beta_n}(x)) + \sqrt{n}(F^A_{\beta_n}(x)-F_{\beta_n}^B(x)) | > c).
\end{align}
We first consider the case $F^A\neq F^B$. Let $x\in\RR^p$ be such that
$F^A(x)\neq F^B(x)$. Then $\lim_n
F^A_{\beta_n}(x)-F_{\beta_n}^B(x)\neq0$, so
$|\sqrt{n}(F^A_{\beta_n}(x)-F_{\beta_n}^B(x)|$ tends to infinity as
$n$ tends to infinity. By the central limit theorem,
$\sqrt{n}(\FF_{\beta_n}^A(x)-F^A_{\beta_n}(x))$ converges in
distribution. Therefore, (\ref{eq:OnePointBound}) yields the result.

Next, consider the case $F^A=F^B$ and
$\Gamma_1(A)\neq\Gamma_1(B)$. Let $x\in\RR^p$ be such that
$\Gamma_1(A)(I_x)\neq\Gamma_1(B)(I_x)$. Similarly to the proof of Theorem
\ref{theorem:GrossErrorAsymptotics}, define $\mu_0=\zeta$,
$\mu_1=(\xi-\zeta)/\|\xi-\zeta\|_\infty$, $S_k
=\{\alpha\in\{0,1\}^p\left| \sum_{i=1}^p \alpha_i=k\right.\}$ and also
$\Gamma_k(A) = \sum_{\alpha\in
  S_k}L_A(\otimes_{i=1}^p\mu_{\alpha_i})$. Note that $\Gamma_k(A)$
with $k=1$ corresponds to (\ref{eq:Gamma1Special}). Then, we have
\begin{align}
  L_A(P_e(\beta)^{\otimes p})
  =&\sum_{k=0}^p \beta^k \|\xi-\zeta\|_\infty^k \Gamma_k(A),
\end{align}
see (\ref{eq:LAPeBetaGammaExp}). In particular, we obtain
\begin{align}
  F^A_\beta(x) - F^B_\beta(x)
  =&L_A(P_e(\beta)^{\otimes p})(I_x)
   -L_B(P_e(\beta)^{\otimes p})(I_x)\notag\\
  =&\sum_{k=1}^p \beta^k\|\xi-\zeta\|_\infty^k(\Gamma_k(A)(I_x)-\Gamma_k(B)(I_x)),
\end{align}
where we have used that $\Gamma_0(A)=\Gamma_0(B)$, since
$F^A=F^B$. Since $\beta_n=n^{-\rho}$, we obtain
\begin{align}
  \sqrt{n}(F^A_{\beta_n}(x)-F_{\beta_n}^B(x))
  =&\sum_{k=1}^p n^{1/2-k\rho}\|\xi-\zeta\|_\infty^k(\Gamma_k(A)(I_x)-\Gamma_k(B)(I_x)).\label{eq:ScalednRho}
\end{align}
As $\rho<1/2$, $1/2-\rho>0$. As
$\|\xi-\zeta\|_\infty(\Gamma_1(A)(I_x)-\Gamma_1(B)(I_x))\neq0$, we conclude
that as $n$ tends to infinity, the term corresponding to $k=1$ in the
above tends to infinity in absolute value. Since the right hand side 
of (\ref{eq:ScalednRho}) is a sum with finitely many terms, where the
remaining terms are of lower degree in $n$, we conclude that $|\sqrt{n}(F^A_{\beta_n}(x)-F_{\beta_n}^B(x))|$ tends to infinity as
$n$ tends to infinity. As in the previous case, since the ordinary
central limit theorem shows that
$\sqrt{n}(\FF_{\beta_n}^A(x)-F^A_{\beta_n}(x))$ converges in
distribution, (\ref{eq:OnePointBound}) yields the result.
\hfill$\Box$

\section*{Acknowledgements}

We thank the anonymous reviewer and the editor for their comments and
suggestions, which led to considerable improvements of the paper.

\bibliographystyle{amsplain}

\bibliography{full}

\end{document}